\documentclass[reqno,a4paper]{amsart}
\usepackage[utf8]{inputenc}
\usepackage[T1]{fontenc}
\usepackage{amssymb,amsthm,amsmath}
\usepackage{calrsfs}
\usepackage{hyperref}
\usepackage{mathbbol}
\usepackage{mathabx}
\usepackage[all]{xy}
\usepackage{fancyvrb}
\usepackage{fvextra}
\usepackage{seqsplit}

\newdir{>}{{}*:(1,-.2)@^{>}*:(1,+.2)@_{>}}
\newdir{<}{{}*:(1,+.2)@^{<}*:(1,-.2)@_{<}}
 
\theoremstyle{plain}
\newtheorem{theorem}[subsection]{Theorem}
\newtheorem{proposition}[subsection]{Proposition}
\newtheorem{corollary}[subsection]{Corollary}	
\newtheorem{lemma}[subsection]{Lemma}

\theoremstyle{remark}
\newtheorem{remark}[subsection]{Remark}

\newcommand{\noproof}{\hfill \qed}
\newcommand{\comp}{\raisebox{0.2mm}{\ensuremath{\scriptstyle{\circ}}}}
\newcommand{\defn}{\textbf}

\newcommand{\V}{\ensuremath{\mathcal{V}}}

\newcommand{\K}{\ensuremath{\mathbb{K}}}
\newcommand{\N}{\ensuremath{\mathbb{N}}}
\newcommand{\Q}{\ensuremath{\mathbb{Q}}}
\newcommand{\Z}{\ensuremath{\mathbb{Z}}}

\newcommand{\gl}{\mathfrak{gl}}

\newcommand{\Grp}{\ensuremath{\mathsf{Grp}}}
\newcommand{\Pt}{\ensuremath{\mathsf{Pt}}}
\newcommand{\Lie}{\ensuremath{\mathsf{Lie}}}
\newcommand{\Leib}{\ensuremath{\mathsf{Leib}}}
\newcommand{\Liek}{\ensuremath{\mathsf{Lie}}_{\K}}
\newcommand{\Alg}{\ensuremath{\mathsf{Alg}}}
\newcommand{\AAlg}{\ensuremath{\mathsf{AAlg}_\K}}
\newcommand{\qLie}{\ensuremath{\mathsf{qLie}}}
\newcommand{\qLiek}{\ensuremath{\mathsf{qLie}}_{\K}}
\newcommand{\Vect}{\ensuremath{\mathsf{Vect}}}
\newcommand{\Mag}{\ensuremath{\mathsf{Mag}}}
\newcommand{\Set}{\ensuremath{\mathsf{Set}}}
\newcommand{\Bool}{\ensuremath{\mathsf{Bool}}}
\newcommand{\kar}{\ensuremath{\mathit{char}}}
\newcommand{\Der}{\ensuremath{\mathit{Der}}}
\newcommand{\Aut}{\ensuremath{\mathit{Aut}}}
\newcommand{\Act}{\ensuremath{\mathit{Act}}}
\newcommand{\Hom}{\ensuremath{\mathit{Hom}}}
\newcommand{\Rep}{\ensuremath{\mathit{Rep}}}
\newcommand{\SpltExt}{\ensuremath{\mathit{SpltExt}}}

\newenvironment{tfae}
{
\begin{enumerate}}
{\end{enumerate}}

\newcommand{\LACC}{{\rm (LACC)}}
\newcommand{\Singular}{{\sc Singular}}

\begin{document}

\title{Algebras with representable representations}

\author{X.~García-Martínez}
\author{M.~Tsishyn}
\author{T.~Van~der Linden}
\author{C.~Vienne}

\email{xabier.garcia.martinez@uvigo.gal}
\email{matsvei.tsishyn@ulb.ac.be}
\email{tim.vanderlinden@uclouvain.be}
\email{corentin.vienne@uclouvain.be}

\address[Xabier García-Martínez]{Departamento de Matemáticas, Esc.\ Sup.\ de Enx.\ Informática, Campus de Ourense, Universidade de Vigo, E--32004 Ourense, Spain---Faculty of Engineering, Vrije Universiteit Brussel, Pleinlaan 2, B--1050 Brussel, Belgium}
\address[Matsvei Tsishyn]{Institut de Recherches Interdisciplinaires et de Développements en Intelligence Artificielle (IRIDIA), Université Libre de Bruxelles, Campus du Solbosch -- CP 194/06, Avenue F.\ D.\ Roosevelt 50, B--1050 Bruxelles, Belgium}
\address[Tim Van~der Linden, Corentin Vienne]{Institut de
Recherche en Math\'ematique et Physique, Universit\'e catholique
de Louvain, che\-min du cyclotron~2 bte~L7.01.02, B--1348
Louvain-la-Neuve, Belgium}

\thanks{The first author is a Postdoctoral Fellow of the Research Foundation–Flanders (FWO) and was supported by Ministerio de Economía y Competitividad (Spain), with grant number MTM2016-79661-P. The third author is a Research Associate of the Fonds de la Recherche Scientifique--FNRS. The fourth author is supported by the {Fonds Thelam} of the {Fondation Roi Baudouin}. Computational resources have been provided by the Consortium des Équipements de Calcul Intensif (CÉCI), funded by the Fonds de la Recherche Scientifique de Belgique (F.R.S.-FNRS) under Grant No.~2.5020.11 and by the Walloon Region.}

\begin{abstract}
Just like group actions are represented by group automorphisms, Lie algebra actions are represented by derivations: up to isomorphism, a split extension of a Lie algebra $B$ by a Lie algebra $X$ corresponds to a Lie algebra morphism $B\to \Der(X)$ from $B$ to the Lie algebra $\Der(X)$ of derivations on~$X$. In this article, we study the question whether the concept of a derivation can be extended to other types of non-associative algebras over a field $\K$, in such a way that these generalised derivations characterise the $\K$-algebra actions. We prove that the answer is~no, as soon as the field $\K$ is infinite. In fact, we prove a stronger result: already the representability of all \emph{abelian} actions---which are usually called \emph{representations} or \emph{Beck modules}---suffices for this to be true. Thus we characterise the variety of Lie algebras over an infinite field of characteristic different from $2$ as the only variety of non-associative algebras which is a non-abelian category with representable representations. This emphasises the unique role played by the Lie algebra of linear endomorphisms $\gl(V)$ as a representing object for the representations on a vector space~$V$.
\end{abstract}


\maketitle

\section*{Introduction}

Groups act by automorphisms, and Lie algebras act by derivations: thus, via the semi-direct product construction, actions are equivalent to split extensions, while (up to isomorphism) any split extension of a group $B$ by a group $X$ corresponds to a group homomorphism $B\to \Aut(X)$, and likewise for Lie algebras, with $\Aut(X)$ replaced by the Lie algebra $\Der(X)$ of derivations on $X$. In this article, we study the question whether the concept of a derivation can be extended from Lie algebras to other types of non-associative algebras, in such a way that these generalised derivations characterise the algebra actions.

The situation sketched here has a categorical description due to F.~Borceux, G.~Janelidze and G.~M.~Kelly called \emph{action representability}~\cite{BJK2,BJK}, which is expressed by saying that split extensions by an object $X$ are representable by an object $[X]$. This means that we have a bijection $\SpltExt(B,X) \cong \Hom(B,[X])$, natural in~$B$, between the set of isomorphism classes of split extensions of~$B$ by~$X$ and the set of morphisms from $B$ to $[X]$. The object $[X]$ corresponds to the group of automorphisms $\Aut(X)$ in the case of groups, or the algebra of derivations $\Der(X)$ in the case of Lie algebras. In~\cite{BJK2,BJK}, other examples such as Boolean rings are studied, as well as equivalent descriptions of the condition. A slightly different, more ``object-wise'' approach appears in~\cite{CDLactors}, where also an overview of the relevant literature is given.

In the present article we work in the general setting of varieties of algebras over a field---``variety'' in the sense of universal algebra, which is different from its use in algebraic geometry. For a variety of non-associative algebras $\V$ over a field~$\K$, we seek an algebra~$[X]$ in~$\V$ that represents the split extensions of a given $\V$-algebra~$X$. We show that such an algebra $[X]$ cannot exist for each $X$ in $\V$, as soon as the field $\K$ is infinite, unless $\V$ is either $\Lie_\K$, $\qLie_\K$, or an abelian category. We thus characterise the varieties of (quasi-)Lie algebras over an infinite field as the only varieties of non-associative algebras that form a non-abelian action representable category.

Our method actually proves a significantly stronger result: it turns out that there is no loss in reducing the representability condition to \emph{abelian} actions. These are usually called \emph{representations} or \emph{Beck modules}~\cite{Beck} and in the present context amount to actions on an abelian algebra---that is, an algebra whose multiplication is zero, so that the identity $xy=0$ holds; see below for a detailed explanation. We say that $\V$ has \defn{representable re\-pre\-senta\-tions} when for each abelian algebra~$X$, the contravariant functor $\Rep(\_,X)\colon \V\to \Set$ that sends an algebra $B$ to the set of $B$-module structures on $X$ is a re\-pre\-senta\-ble functor, which means that it is naturally isomorphic to $\Hom(\_,[X])$ for some $\V$-algebra $[X]$. We show that, when it is a non-abelian category, such a variety~$\V$ is either $\Lie_\K$ or $\qLie_\K$, in which case $[X]$ is necessarily $\gl(X)$, the Lie algebra of linear endomorphisms of the vector space $X$. The object $\gl(X)$ is typical for (quasi\nobreakdash-)Lie algebras, in the sense that no other non-abelian variety of $\K$-algebras has an object~$[X]$ representing the module structures on $X$. In other words, changing the variety breaks $\gl(X)$ beyond repair.

\subsection*{Overview of the text}
The article is organised as follows. In Section~\ref{Section Preliminaries}, we recall some basic definitions and results concerning varieties of non-associative algebras. We introduce \emph{actions} and the condition that they are \emph{representable} in the context of semi-abelian categories. We explain what this amounts to in varieties of non-associative algebras. Then we discuss the concept of a representation.

In Section~\ref{Section AC}, we prove that for a variety, representability of its representations implies a condition called \emph{algebraic coherence}. This allows us to work with the $\lambda / \mu$-rules, which are identities of degree three, useful in what follows. In Section~\ref{Section two degree}, the idea is to show that some identities of degree two---necessarily commutativity or anticommutativity---follow from representability of representations. In order to get these results, we use two computer programs, which probably makes Section~\ref{Section two degree} the most innovative section of this paper. In Section~\ref{Anticommutativity}, the goal is to observe that only anticommutativity is consistent with representability of representations, as long as the variety is non-abelian. From this we deduce that the $\lambda / \mu$-rules can be reduced to the Jacobi identity. In Section~\ref{Subvarieties}, we answer a final question which arises naturally from the previous sections. A~priori, anticommutativity and the Jacobi identity need not be the only identities satisfied in a variety with representable representations. We prove that there is indeed nothing else.

In Section~\ref{Conclusion}, we conclude with Theorem~\ref{AR implies Lie}, the goal and main result of this paper. We position our work within the context of similar results in the literature, and discuss some open problems.

\section{Preliminaries}\label{Section Preliminaries}

\subsection{Varieties of algebras}

An \defn{algebra} $A$ on a field $\K$ is a $\K$-vector space equipped with a bilinear map $A \times A\rightarrow A\colon (x,y)\mapsto xy=x\cdot y=[x,y]$ called the \emph{multiplication} or \emph{bracket}. A morphism of algebras is therefore a $\K$-linear map $f\colon A \rightarrow B$ which preserves this multiplication. This determines the category $\Alg_\K$ of algebras over~$\K$. We remark that for now we do not require the multiplication to be associative.

We call a collection of $\K$-algebras a \defn{variety} of (non-associative) algebras over~$\K$ if the collection contains all the algebras satisfying a chosen set of polynomial equations. For example, the variety $\Liek$ of \defn{Lie algebras} corresponds to the collection of all algebras over $\K$ satisfying the Jacobi identity ($x(yz)+ y(zx) + z(xy)=0$) and alternativity ($xx=0$), while $\qLiek$ denotes the variety of \defn{quasi-Lie algebras} where alternativity is replaced by anticommutativity ($xy+ yx =0$). Of course when $ \kar(\K) \neq2$, quasi-Lie algebras and Lie algebras are the same. However, for characteristic $2$, alternativity implies anticommutativity but the converse is not true. In this case, the variety of Lie algebras is strictly smaller than the variety of quasi-Lie algebras.

We say that an algebra is \defn{abelian} if $xy=0$ holds for every $x$ and $y$ in the algebra. The variety of abelian $\K$-algebras is trivially equivalent to the (abelian) category $\Vect_{\K}$ of vectors spaces over $\K$. Each variety of $\K$-algebras contains all abelian algebras over $\K$, and these are precisely those algebras that admit an internal abelian group structure in the category $\V$.

\begin{remark}
Any variety $\V$ of non-associative algebras is a Janelidze-M\'arki-Tholen \emph{semi-abelian} category~\cite{Janelidze-Marki-Tholen}. Indeed, any variety of algebras is in particular a \emph{variety of $\Omega$-groups} in the sense of Higgins~\cite{Higgins}.
\end{remark}

We recall some additional concepts in order to cite two theorems we shall need later on. Let $F\colon \Set \rightarrow \Alg_\K$ be the \emph{free algebra} functor sending a set $S$ to the free algebra over $\K$ generated by the elements of $S$. We recall that it is a left adjoint functor (its right adjoint being the forgetful functor) which factorises through the \emph{free magma} functor $M\colon \Set \rightarrow \Mag$ (sending $S$ to the magma $M(S)$ of nonempty words in $S$) and the \emph{magma algebra} functor $\K[\_]\colon {\Mag \rightarrow \Alg_\K}$. A (non-associative) \defn{polynomial} $\varphi$ on a set $S$ is an element of $\K[M(S)]$. We say that $\varphi$ is a \defn{monomial} if it is a scalar multiple of an element of $M(S)$. The \defn{type} of a monomial $\varphi = \varphi(x_1, \dots , x_n)$ is a $n$-tuple $(k_1, \dots, k_n)\in \N^n$ where $k_i$ is the number of times $x_i$ appears in $\varphi$. Its \defn{degree} is the number $k_1+\cdots+k_n$. A polynomial is said to be \defn{homogeneous} if all its monomials are of the same type. We can decompose every polynomial into its homogeneous components.

\begin{proposition}\label{Comm or Anticomm}
If $\V$ is a variety of non-associative $\K$-algebras satisfying a non-trivial homogeneous identity of degree $2$, then $\V$ is either a subvariety of the variety of commutative algebras, or a subvariety of the variety of anticommutative algebras.
\end{proposition}
\begin{proof}
We want to prove that essentially the only two possibilities are commutativity $(xy=yx)$ and anticommutativity $(xy=-yx)$. In fact, alternativity is also a non-trivial identity. But since $xx=0$ implies anticommutativity, we can assume $\varphi$ to be $xy + \lambda yx =0$ for some $\lambda \in \K$ without any loss of generality. This implies that $xy= -\lambda yx$ and thus $(1-\lambda^2 )xy=0$. Therefore, either $\V$ is an abelian variety, or $\lambda =\pm 1$, which proves the result.
\end{proof}

\begin{theorem}[\cite{Shestakov, Bahturin}]\label{Theorem Homogeneity}
If $\V$ is a variety of algebras over an \emph{infinite} field, then all of its identities are of the form $\phi(x_{1},\dots,x_{n})=0$, where $\phi(x_{1},\dots,x_{n})$ is a polynomial, each of whose homogeneous components $\psi(x_{i_{1}},\dots, x_{i_{m}})$ again gives rise to an identity $\psi(x_{i_{1}},\dots, x_{i_{m}})=0$ in $\V$. \noproof
\end{theorem}

Unless we specify it otherwise, we shall always assume that $\K$ is an infinite field, so that we can use the previous result. For instance, when $\V$ is a variety of $\K$-algebras that satisfies a non-trivial identity involving monomials of degree $2$, then Proposition~\ref{Comm or Anticomm} applies and tells us that~$\V$ consists of either commutative or anticommutative algebras.

A homogeneous polynomial is \defn{multilinear} if its monomials are of the type $(1, 1, \dots, 1)$. A \defn{multilinear identity} is therefore an identity $\varphi(x_1,\dots,x_n)=0$ where $\varphi$ is a multilinear polynomial.

\begin{theorem}[Theorem~3 in Section~4.2 of~\cite{Bahturin}]\label{Theorem Linearity General}
	In a subvariety of $\Liek$, any nontrivial identity has a nontrivial multilinear consequence.\noproof
\end{theorem}

\subsection{Actions of algebras}\label{Actions}
Let $\V$ be a semi-abelian category \cite{Janelidze-Marki-Tholen} and $B$ be an object in $\V$. We write $\Pt_B(\V)$ for the category of points over $B$ whose objects are triples $(A,p,s)$ where $A$ is an object in $\V$ and $p\colon A\rightarrow B$ is a split epimorphism with a given section $s$. It is well known \cite{Bourn-Janelidze:Semidirect} that the functor $K\colon \Pt_B(\V) \rightarrow \V$ sending a point $(A,p,s)$ over $B$ to the kernel of $p$ is monadic. The corresponding monad on~$\V$ is the functor $B \flat (\_)\colon \V \to \V\colon X \mapsto B\flat X$ where the object $B\flat X$ is the kernel of $(1_B,0)\colon B + X \to B$, together with certain natural transformations $\eta^B\colon 1_{\V}\Rightarrow B \flat (\_)$ and $\mu^B\colon B \flat (B \flat (\_))\Rightarrow B \flat (\_)$. Here $\mu^{B}_X$ is a restriction of the codiagonal $(B+B)+X\to B+X$ and $\eta^{B}_X$ sends an element of~$X$ to itself, considered as an element of $B\flat X$. A \defn{$B \flat (\_)$-algebra} $(X,\xi)$ is an object $X$ together with a morphism $\xi\colon B \flat X \rightarrow X$ called an \defn{action} of~$B$ on~$X$, such that the diagrams
\begin{align*}
\vcenter{\xymatrix{X \ar@{=}[rd] \ar[r]^-{\eta^{B}_X} & B\flat X \ar[d]^-{\xi}\\
& X
}}
\qquad\text{and}\qquad
\vcenter{\xymatrix{
B\flat(B\flat X) \ar[d]_-{1_{B}\flat\xi} \ar[r]^-{\mu^{B}_X} & B\flat X \ar[d]^-{\xi}\\
B\flat X \ar[r]_-{\xi} & X
}}
\end{align*}
commute. We write $\Act(B,X)$ for the set of actions of~$B$ on $X$. One equivalent way of viewing actions uses split extensions:

\begin{lemma}[\cite{Bourn-Janelidze:Semidirect, BJK2}]\label{actsplt}
Given objects $B$ and $X$ in a semi-abelian category, there is a bijection $\SpltExt(B,X) \cong \Act(B,X)$ between the set of isomorphism classes of split extensions of $B$ by $X$ and the set of actions of $B$ on $X$.\noproof
\end{lemma}
In fact, for any action $\xi\colon B \flat X \rightarrow X$ there is a split extension of $B$ by~$X$
\[
\xymatrix{0 \ar[r] & X \ar[r]^-k & H \ar@<.5ex>[r]^-p & B \ar@<.5ex>[l]^-s \ar[r] & 0}
\]
and vice versa. The object $H$ in this split extension is called the \defn{semi-direct product} of $B$ with $(X,\xi)$, written $B \ltimes_\xi X$. In a variety of non-associative algebras, what is its structure? 

First, since it is the kernel of the morphism $(1_B,0)\colon {B+X\to B}$, the object~${B \flat X}$ consists of polynomials with variables in $B$ and $X$ which can be written in a form where all of their monomials contain variables in $X$. 

Then, by applying the forgetful functor $U\colon \V \rightarrow \Vect$, we observe that, as a vector space, $B \ltimes_\xi X$ is the direct sum/cartesian product $B \oplus X\cong B\times X$. Next, routine calculations show us that the multiplication has to be
\begin{align}\label{sdp}
(b,x) \cdot (c,y) = (bc,\,{}^by + x^c + xy),
\end{align}
where $^by$ and $x^c$ are notations for the image by $\xi$ in $X$ of $by$ and $xc$ respectively. Sometimes, for the sake of simplicity, we will just write $by$ or $xc$. We remark that starting with a \emph{morphism} $\xi\colon B \flat X \rightarrow X$ (not necessary an action) we can still define a semi-direct product $\K$-algebra by~\eqref{sdp}. The question is now, whether this algebra belongs to the variety $\V$.

\begin{lemma}\label{sdp in V}
Let $\V$ be a variety of non-associative algebras over a field $\K$, let $B$ and $X$ be two algebras and $\xi\colon B \flat X \rightarrow X$ a morphism in $\V$ such that $\xi\circ\eta^B_X=1_X$. Then $\xi$ is an action in $\V$ if and only if $B \ltimes_\xi X$ is in $\V$.
\end{lemma}
\begin{proof}
The first implication directly comes from the equivalence in Lemma~\ref{actsplt}. For the converse we construct the split extension
\[
\xymatrix{0 \ar[r] & X \ar[r]^-k & B \ltimes_\xi X \ar@<.5ex>[r]^-p & B \ar@<.5ex>[l]^-s \ar[r] & 0}
\]
in $\V$, where $k(x)=(0,x)$, $p(b,x)=b$ and $s(b)=(b,0)$. Lemma~\ref{actsplt} then gives us an action $\chi\colon B \flat X \rightarrow X$ in $\V$ which is, by construction, the only morphism satisfying $k \comp \chi = (s,k) \comp l$ where $l$ is the kernel of $(1_B,0)\colon B + X \to B $. But $\xi$ obviously satisfies this condition, and therefore coincides with $\chi$.
\end{proof}

This description of actions, analogous to Theorem 2.5 of \cite{CDLactors}, will be useful all along this paper, whenever we wish to check that an action is well defined. For example, let $\V$ be the variety of commutative algebras over an arbitrary field $\K$, and let $V$ be an algebra in this variety. Then $V$ acts on itself by $^vw=v\cdot w$ and $v^w=v\cdot w$. Indeed, the semi-direct product $\K$-algebra is in $\V$, since for all $v$, $v'$, $w$, $w'\in V$ we have $(v,w)\cdot (v',w') =(v',w')\cdot (v,w)$.

\subsection{Action representability}
Let $\V$ be a semi-abelian category. For a fixed object $X$ in the category, $\Act(\_,X)\colon \V \rightarrow \Set$ defines a contravariant functor sending an object $B\in \V$ to the set of actions of $B$ on $X$. We say that the category~$\V$ is \defn{action representable} when for each object $X$ in $\V$ the functor $\Act(\_,X)$ is representable \cite{BJK2}. This means that there exists an object $[ X]$, called the \defn{actor} of $X$, in $\V$ and a natural isomorphism
\begin{align*}
\Act(\_,X) \cong \Hom(\_,[X]).
\end{align*}
An interesting question is the role of the actor in concrete examples: What is then its structure? In $\Grp$, the actor of an object $X$ is the automorphism group $\Aut(X)$; in $\Lie_\K$, the actor of $X$ is the Lie algebra of its derivations $\Der(X)$. More examples are studied in \cite{BJK2}, and new examples continue to be studied, as for crossed modules (in \cite{Ramasu}) and for cocommutative Hopf algebras (in \cite{GKV2}).

\begin{proposition}\label{jsp}
Let $\V$ be a variety of non-associative algebras over $\K$. Then $\V$ is action representable if and only if for every object $X$ in $\V$ there exists an object $[X]$ acting on $X$ which has the following property: for any object $B$ with an action on $X$ there is a unique morphism $\varphi\colon B \rightarrow [X]$ such that $^bx= {}^{\varphi (b)}x$ and $x^b=x^{\varphi(b)}$ for every $x\in X$ and $ b\in B$. 
\end{proposition}
\begin{proof}
This is a reformulation of Proposition 3.1 from \cite{CDLactors} for the context of varieties of non-associative algebras, which are all categories of \emph{groups with operations} in the sense of~\cite{Porter-GroupsWithOperations}.
\end{proof}

In \cite{BJK2}, it is explained that action representability in so-called \emph{locally well-present\-able} semi-abelian categories is equivalent to the condition that $\Act(\_,X)$ preserves binary coproducts for every object $X$. Actually, the isomorphism $\Act(\_,X) \cong \Hom(\_,[X])$ easily implies that $\Act(B + B',X) \cong \Act(B,X) \times \Act(B',X)$, which can be understood as the condition that ``having an action of $B+B'$ on $X$ is the same as having two actions on $X$, one of $B$ and the other of $B'$.'' 

From this observation, we may now deduce an interesting and useful technique. Indeed, by Lemma~\ref{sdp in V}, the object $(B+B') \ltimes X$ is in $\V$ and as such should satisfy the identities characterizing $\V$, whenever the variety is action representable. We remark that this stays true with a finite number of algebras acting on $X$. Checking this condition is an important proof method used throughout this article. Let us give a concrete example of how this works, by reproving the known result that in the variety of associative algebras, the actor cannot always exist:

\begin{proposition}\label{Ass not AR}
Let $\K$ be a field, and $\AAlg$ the variety of associative algebras over $\K$ (defined by $(xy)z=x(yz)$). Then $\AAlg$ is not action representable.
\end{proposition}
\begin{proof}
Suppose that $\AAlg$ is action representable. Let $B^1$ and $B^2$ be two algebras with trivial multiplication respectively generated by the elements $b^1$ and $b^2$. Let~$X$ be the algebra with trivial multiplication generated by the elements $x$, $y^1$, $y^2$, $z^{12}$ and $z^{21}$. Those algebras are in $\AAlg$. We define the actions of $B^i$ on $X$ by:
\begin{align*}
b^ix=xb^i=y^i,
\quad
b^i y^j = y^j b^i =
\begin{cases}
z^{ij} & \text{if $i \neq j$} \\
0 & \text{otherwise}
\end{cases}
\quad
\text{and}
\quad b^i z^{jk}=z^{jk}b^i =0,
\end{align*}
where $i$, $j$, $k\in\{1,2\}$ and $j\neq k$. These choices determine morphisms of $\V$-algebras $\xi^i\colon B^i\flat X\to X$, and Lemma~\ref{sdp in V} implies that those are well-defined actions. Since we supposed $\AAlg$ action representable, also the algebra $B^1 + B^2$ acts on~$X$. Yet $(B^1+B^2)\ltimes X$ is not in $\AAlg$, which is a contradiction. Indeed, we have that $b^1(xb^2)=z^{12} \neq z^{21} = (b^1x) b^2$.
\end{proof}

\begin{theorem}\label{Lie is AR}
$\Lie_\K$ and $\qLie_\K$ are action representable for any field~$\K$.
\end{theorem}
\begin{proof}
For Lie algebras this is a well-known result proved in \cite{BJK, BJK2}. The case of quasi-Lie algebras, for a field of characteristic $2$, is not explicit there, but the essence of the proof stays valid. In any case, the actor of an object $X$ in one of those categories is the algebra of derivations $\Der(X)$. 
\end{proof}

\subsection{Representations and their representability}
Given an object $B$ in $\V$, a~\defn{Beck module~\cite{Beck} over $B$}, also called a \defn{$B$-module} or a \defn{representation of~$B$}, is an abelian group object in the slice category $(\V\downarrow B)$. As is well known (and explained in detail for instance in~\cite{HVdL}), in a semi-abelian category that satisfies the so-called \emph{Smith is Huq} condition, a $B$-module structure on a (necessarily abelian) object~$X$ of~$\V$ is the same thing as a $B$-action on $X$---an ``abelian action''. Combining the results of~\cite{MFVdL} and \cite{MM} with the fact that any variety of non-associative $\K$-algebras is a \emph{category of groups with operations} in the sense of~\cite{Porter-GroupsWithOperations}, we see that this interpretation holds in the context where we are working. So here, $\Rep(B,X)=\Act(B,X)$ when~$X$ is an abelian object.

In the case of Lie algebras we regain the classical concept of a Lie algebra re\-presenta\-tion, which is often defined as a morphism $B\to \gl(X)$, where $\gl(X)$ is the Lie algebra of linear endomorphisms of the vector space $X$. For $f$, $g\colon X\to X$, the bracket $[f,g]=f\comp g-g\comp f$ is a Lie algebra structure on $\Hom(X,X)$, because composition of endomorphisms is associative.

As we shall see below, it makes sense to restrict the representability condition from actions to representations, as follows: we say that $\V$ has \defn{representable re\-pre\-senta\-tions} when for each abelian algebra $X$, the contravariant functor
\[
\Rep(\_,X)=\Act(\_,X)\colon \V\to \Set
\]
that sends an algebra $B$ to the set of $B$-module structures on $X$ is a re\-pre\-senta\-ble functor. 

For instance, if $\V=\Lie_{\K}$, then $\Rep(\_,X)$ is represented by $\gl(X)$, which means $\Rep(\_,X)\cong\Hom(\_,\gl(X))$. On the other hand, the proof of Proposition~\ref{Ass not AR} can be used to show that the variety of associative algebras does not have representable representations.

This condition is clearly weaker than action representability; it will, however, turn out to be sufficient for our purposes. Of course, everything we said above about actions stays valid for representations.

\section{Algebraic coherence}\label{Section AC}
A first step towards the main result of this paper is proving that in the context of non-associative algebras, action representability/representability of representations implies a condition called \emph{algebraic coherence}~\cite{acc}. The reason we want this, comes from the fact that an algebraically coherent variety satisfies some identities of degree three, useful in the next sections. From \cite{GM-VdL2} we have the following characterisation:

\begin{theorem}\label{Theorem AC iff Orzech}
Let $\K$ be an infinite field. If $\V$ is a variety of non-associative $\K$-algebras, then the following conditions are equivalent:
\begin{tfae}
\item $\V$ is \defn{algebraically coherent};
\item $\V$ is an \emph{Orzech category of interest}~\cite{Orzech};
\item $\V$ is a \defn{$2$-variety}: for any ideal $I$ of an algebra $A$, the subalgebra $I^2$ of $A$ is again an ideal;
\item there exist $\lambda_{1}$, \dots, $\lambda_{8}$, $\mu_{1}$, \dots, $\mu_{8}$ in $\K$ such that
\begin{align*}
z(xy)=
\lambda_{1}(zx)y&+\lambda_{2}(xz)y+
\lambda_{3}y(zx)+\lambda_{4}y(xz)\\
&+\lambda_{5}(zy)x+\lambda_{6}(yz)x+
\lambda_{7}x(zy)+\lambda_{8}x(yz)
\end{align*}
and 
\begin{align*}
(xy)z=
\mu_{1}(zx)y&+\mu_{2}(xz)y+
\mu_{3}y(zx)+\mu_{4}y(xz)\\
&+\mu_{5}(zy)x+\mu_{6}(yz)x+
\mu_{7}x(zy)+\mu_{8}x(yz)
\end{align*}
are identities in $\V$. \noproof
\end{tfae}
\end{theorem}

We call those two identities together \defn{the $\lambda/\mu$-rules}. One should understand them as some ``general associativity rules''. We remark that the case where $\lambda_1=1=\mu_8$ and the other coefficients are zero is just associativity. It is easy to see that the Jacobi identity is another particular case of the $\lambda/\mu$-rules.

\begin{theorem}\label{RR to AC}
Let $\K$ be an infinite field and $\V$ be a variety of non-associative algebras over $\K$. If $\V$ has representable representations, then it is an algebraically coherent category.
\end{theorem}
\begin{proof}
Consider the free non-associative algebra on the set $\{b, b', x\}$. We quotient it by the ideal $I$ generated by $xb$, $bx$, $b'x$, $xb'$, $xx$, $bb$, $b'b'$ and all monomials of degree~$3$ or higher in which one of the variables $b$, $b'$, $x$ is repeated. We reflect the result into the variety~$\V$ by dividing out the ideal generated by the identities of~$\V$, and obtain an algebra~$Q$. Since $\K$ is an infinite field, by Theorem~\ref{Theorem Homogeneity} the sub--vector space generated by the classes of the elements of the form $(bb')x$ (all possible permutations and bracketings) consists entirely of classes of linear combinations of such elements of degree $3$. We call this vector space $A$, and view it as an abelian $\K$-algebra.

We write $B$ for the abelian $\V$-algebra generated by $b$, $B'$ for the abelian $\V$-algebra generated by $b'$ and $C$ for the free vector space generated by~$c$, viewed as an abelian $\K$-algebra. In $\V$, we then define a representation of $B + B'$ on the abelian algebra $V=A\times C$ by ${}^ba = 0 = a^b$, ${}^{b'}a = 0 = a^{b'}$, ${}^{bb'}a = {}^{b'b}a = 0 = a^{bb'} = a^{b'b}$ where $a$ is any element of $A$, ${}^bc = 0 = c^b$, ${}^{b'}c = 0 = c^{b'}$ and ${}^{bb'}c = (bb')x$, ${}^{b'b}c = (b'b)x$, $c^{bb'} = x(bb')$, $c^{b'b} = x(b'b)$. We need to verify that this does indeed determine an action. To do so, let us give an explicit description of the corresponding $\V$-algebra morphism $\xi\colon (B+B')\flat V\to V$. We first consider four $\V$-algebra morphisms $B+B'+A+C\to Q\times C$, one which sends $\theta(b,b',a,c)$---where $\theta$ is a polynomial in $n+3$ variables and $a\in A^n$ for some $n\in \N$---to $(\theta(b,b',0,x),0)$, and three others, which send it to $(\theta(0,0,a,0),0)$, $(\theta(0,0,0,x),0)$ and $(0,\theta(0,0,0,c))$, respectively. We combine them into the $\K$-linear map
\begin{align*}
B+B'+A+C&\to Q\times C\colon\\
\theta(b,b',a,c)&\mapsto (\theta(b,b',0,x)-\theta(0,0,0,x)+\theta(0,0,a,0),\theta(0,0,0,c)).	
\end{align*}
Since this map vanishes on monomials in~$a$ and~$c$ that contain at least one element of $A$ and one $c$, it factors through the quotient
\[
{(B+B')+(A+C)\to (B+B')+(A\times C)=(B+B')+V}
\]
to a $\K$-linear map $(B+B')+V\to Q\times C$, which in turn restricts to a $\K$-linear map $(B+B')\flat V\to Q\times C$. It is easy to check that this map factors over the inclusion of $A\times C=V$ into $Q\times C$. The resulting factorisation is a $\K$-linear map $\xi\colon (B+B')\flat V\to V$, which is actually a morphism of $\K$-algebras, because it sends all binary products in $(B+B')\flat V$ to zero, as it should. We may now check that $\xi$ satisfies the axioms of a $B\flat(\_)$-algebra, which follows immediately from the definitions---see~\ref{Actions}. 

Since $\V$ has representable representations in~$\V$, the action $\xi$ is necessarily zero, because that is where it is sent by the canonical isomorphism
\[
\Rep(B+B',V)\to \Rep(B,V)\times \Rep(B',V)
\]
defined by restricting an action of $B+B'$ on $V$ to an action of $B$ on $V$ together with an action of $B'$ on $V$. It follows that $(bb')x$, $(b'b)x$, $x(bb')$ and $x(b'b)$ vanish in~$A$. This means that the identities of $\V$ allow us to write each of them as a linear combination of some degree $3$ elements of the ideal $I$, which can only mean that the $\lambda/\mu$-rules hold in $\V$. In other words, $\V$ is algebraically coherent.
\end{proof}

\begin{remark}
If in Theorem~\ref{RR to AC} the variety $\V$ has representable actions, then an alternate proof may be given, which uses that action representability implies a weaker condition called \emph{action accessibility}~\cite{BJ07}. As it turns out, a variety of non-associative algebras over an infinite field is action accessible if and only if it satisfies the equivalent conditions of Theorem~\ref{Theorem AC iff Orzech}. Indeed, as A.~Montoli proved in~\cite{Montoli}, all \emph{Orzech categories of interest} are action accessible, which yields one implication. The converse makes non-trivial use of Lemma~2.9 in~\cite{CMM2}.
\end{remark}

\section{Identities of degree two}\label{Section two degree}

We proved that from the representability of representations we may deduce certain identities of degree three. The goal of this section is to go for the identities of degree two necessary to obtain Lie algebras. In fact, action representability does not characterise the Jacobi identity alone, since the variety $\Leib_{\K}$ of Leibniz algebras over $\K$ is not action representable in general. This observation comes from Theorem~5.5 of~\cite{Casas-Ladra}, but it is possible to give a proof, similar to the one of Proposition~\ref{Ass not AR}, showing that $\Leib_{\K}$ does not have representable representations.

We are now going to show that any variety of non-associative algebras with representable representations over an infinite field satisfies non-trivial identities of degree two. We start with fields of characteristic zero in Proposition~\ref{two degree char 0}, and continue with fields of prime characteristic in Proposition~\ref{thm:PrimeCharCase}. 

Our main technique is to obtain a system of polynomial equations, which we then prove is inconsistent. To do so, we use computer algebra in two distinct ways: first we need to produce the system of equations itself; due to the complexity of the manipulations needed here, we preferred to do this on a computer, rather than by hand: see below and~\cite{MatveiT-1} where this is explained in further detail. Next, the system of equations thus obtained has to be shown inconsistent. The size of the system makes it impossible to do this by hand; we used the open-source software package \Singular~\cite{DGPS} which gives us an explicit reason for the system's inconsistency. How this is done is explained below. Part of the code used is displayed in Figure~\ref{Code}, 
\begin{figure}
\tiny
\begin{Verbatim}[breaklines=true,numbers=left]
ring r=0,(x(1..8),y(1..8)),dp;
poly f(1) = -1 + (y(1)*y(6)) + (y(2)*y(2)) + (y(3)*x(6)) + (y(4)*x(2));
poly f(2) = (y(1)*y(1)) + (y(2)*y(5)) + (y(3)*x(1)) + (y(4)*x(5)) + y(6);
poly f(3) = (y(1)*y(2)) + (y(2)*y(6)) + (y(3)*x(2)) + (y(4)*x(6)) + y(5);
poly f(4) = (y(1)*y(3)) + (y(2)*y(7)) + (y(3)*x(3)) + (y(4)*x(7)) + y(8);
poly f(5) = (y(1)*y(4)) + (y(2)*y(8)) + (y(3)*x(4)) + (y(4)*x(8)) + y(7);
poly f(6) = (y(1)*y(5)) + (y(2)*y(1)) + (y(3)*x(5)) + (y(4)*x(1));
poly f(7) = (y(1)*y(7)) + (y(2)*y(3)) + (y(3)*x(7)) + (y(4)*x(3));
poly f(8) = (y(1)*y(8)) + (y(2)*y(4)) + (y(3)*x(8)) + (y(4)*x(4));
poly f(9) = -1 + (y(5)*y(5)) + (y(6)*y(1)) + (y(7)*x(5)) + (y(8)*x(1));
\end{Verbatim}

\vspace{-2ex}$\vdots$


\caption{Singular code}\label{Code}
\end{figure}
and the full code and its output are accessible as a set of ancillary files in the arXiv version of our paper.

\begin{proposition}\label{two degree char 0}
Let $\K$ be a field of characteristic $0$ and $\V$ be a variety of non-associative algebras over $\K$. If $\V$ has representable representations, then $\V$ satisfies a non-trivial identity of degree two.
\end{proposition}
\begin{proof}
Let us assume that $\V$ is a variety with representable representations, all of whose non-trivial identities are of degree three or higher. From Theorem~\ref{RR to AC}, we already know that $\V$ satisfies the $\lambda / \mu$-rules. The strategy of this proof is to show that representability of representations forces the coefficients $\lambda_i$ and $\mu_i$ to satisfy a specific system of polynomial equations whose set of solutions is empty (unless some degree two identities are satisfied, but we assumed that it is not the case). This then means that algebraic coherence, and therefore representability of representations, cannot hold in $\V$---which is a contradiction.

In order to do so, we let $X$ be the $79$-dimensional abelian algebra generated by the elements $ \{ \text{$x$, $y^i_\alpha$, $z^{ij}_{\alpha\beta}$, $t^{ijk}_{\alpha\beta\gamma}$}\mid \text{$\alpha,\beta,\gamma \in\{ r,l\}$ and $(i\; j\; k) \in S_3$}\}$, where $r$ stands for \emph{right} and $l$ for \emph{left}, while $i$, $j$, $k\in \{1,2,3\}$ are pairwise non-equal. Furthermore, $B^i$ is the abelian algebra generated by $b^i$, for $i \in \{ 1,2,3\}$. Those algebras are in $\V$ because they are abelian. We now define (abelian) actions of the algebras $B^i$ on $X$:
\begin{align*}
b^ix &=y^i_l & xb^i &=y^i_r \\ 
b^i y^j_\alpha &=
\begin{cases}
z^{ji}_{\alpha l} & \text{if $i \neq j$} \\
0 & \text{otherwise}
\end{cases}
&
y^j_\alpha b^i &= 
\begin{cases}
z^{ji}_{\alpha r} & \text{if $i \neq j$} \\
0 & \text{otherwise}
\end{cases}
\\
b^i z^{jk}_{\alpha\beta}&=
\begin{cases}
t^{jki}_{\alpha \beta l} & \text{if $i \neq j$ and $i\neq k$} \\
0 & \text{otherwise}
\end{cases}
&
z^{jk}_{\alpha\beta}b^i&=
\begin{cases}
t^{jki}_{\alpha \beta r} & \text{if $i \neq j$ and $i\neq k$} \\
0 & \text{otherwise}
\end{cases}
\end{align*}
while $b^i t^{jkm}_{\alpha\beta\gamma}=t^{jkm}_{\alpha\beta\gamma}b^i =0$. Those actions are well defined. Indeed, the $\lambda / \mu$-rules are satisfied since any product of three elements in the semi-direct product $B^i \ltimes X$ vanishes. In order to familiarise ourselves with those actions which will appear again later, we first give some examples of their behaviour in $(B^1+B^2+B^3)\ltimes X$:
\begin{align*}
(b^1(b^2x))b^3 = t^{213}_{ l l r} , \quad ((b^2x)b^3)b^2 = 0, \quad b^1(b^2y^3_r)=t^{321}_{rll}.
\end{align*}
Next, as explained in Section~\ref{Section Preliminaries}, by representability of representations the semi-direct product $(B^1+B^2+B^3)\ltimes X$ is in $\V$ and thus all elements of this semi-direct product have to satisfy the $\lambda / \mu$-rules. We construct an equational system on the $\lambda_i$ and $\mu_i$ whose set of solutions is empty, by checking the $\lambda / \mu$-rules for some specific elements. For example, since $(xb^1)b^2$ is in $(B^1+B^2+B^3)\ltimes X$, decomposing it a first time gives:
\begin{align*}
(xb^1)b^2 = \mu_{1}(b^2x)b^1&+\mu_{2}(xb^2)b^1+
\mu_{3}b^1(b^2x)+\mu_{4}b^1(xb^2) 
\\&+\mu_{5}(b^2b^1)x+\mu_{6}(b^1b^2)x+
\mu_{7}x(b^2b^1)+\mu_{8}x(b^1b^2).
\end{align*}
Then we apply the $\lambda / \mu$-rules again on the monomials where $x$ is isolated and write everything on the same side. This gives us a linear combination $\sum f_k(\lambda_i,\mu_j) z^{ij}_{\alpha\beta}$ which must be equal to zero. Since the $z^{ij}_{\alpha\beta}$ are linearly independent, all the polynomials $f_k\in \Q[\lambda_1, \dots, \lambda_8, \mu_1, \dots, \mu_8]$ vanish. Following this procedure for $(xb^1)b^2$ gives us the eight first polynomials of Figure~\ref{Code} (lines 2--9, written in \Singular~\cite{DGPS} code). To compute the $f_k$ for $k=9$, \dots, $32$, we repeat this procedure on the elements $(b^1x)b^2$, $b^2(b^1x)$ and $b^2(xb^1)$. The details of the computations are done and explained in~\cite{MatveiT-1}.

Next, we use the fact that $(b^1b^2)b^3$ and $b^1(b^2b^3)$ should also satisfy the $\lambda /\mu$-rules since they are elements of $(B^1+B^2+B^3)\ltimes X$. Making the decomposition work on $(b^1(b^2b^3))x$, $x(b^1(b^2b^3))$, $((b^1b^2)b^3)x$ and $x((b^1b^2)b^3)$---again, a procedure explained in~\cite{MatveiT-1}---we find the last 192 polynomials in the system of equations.

Finally, we employ the computer algebra system \Singular~\cite{DGPS}, which uses Gr\"obner bases, to look for a common root of the polynomials. Based on the code in Figure~\ref{Code}, the computer system tells us that $1$ is a linear combination in $\Q[\lambda_1, \dots, \lambda_8, \mu_1, \dots, \mu_8]$ of the polynomials $f_{1}$, \dots, $f_{224}$. The arXiv version of this paper includes an ancillary file containing an explicit way to write $1$ as a linear combination of the $(f_i)_i$ in $\Q[\lambda_1, \dots, \lambda_8, \mu_1, \dots, \mu_8]$. Hence the set of solutions of the polynomial system is empty, which completes the proof.
\end{proof}

\begin{remark}
The parameter \verb+dp+ in line~1 means that degree reverse lexicographical ordering of polynomials is chosen. This choice does not have any impact in the proof other than the efficiency of the computation.
\end{remark}
\begin{remark}
The system of equations in Figure~\ref{Code} is not the smallest inconsistent system of equations involving the $\lambda_i$ and $\mu_i$. For instance, from the output in the ancillary files it is immediately clear that $f_{140}$ can be removed from the system while maintaining its inconsistency. Actually, the system may be reduced significantly: we checked that equations number $1$, $3$, $4$, $6$, $7$, $10$, $11$, $12$, $13$, $14$, $15$, $16$, $17$, $18$, $19$, $20$, $22$, $23$, $25$, $26$, $27$, $28$, $29$, $30$, $31$, $33$, $34$, $35$, $36$, $37$, $38$, $39$, $40$, $41$, $42$, $43$, $44$, $45$, $46$, $47$, $49$, $56$, $82$, $92$ together still form an inconsistent system of polynomial equations. However, giving the explicit polynomials that prove the system's inconsistency becomes harder as its size goes down. 
\end{remark}

In order to have the result for all infinite fields, we extend this to prime characteristics. The proof needs fine-tuning since in characteristic zero rational coefficients are allowed, something which is not the case in prime characteristic.

\begin{proposition}\label{thm:PrimeCharCase}
Let $\K$ be an infinite field of prime characteristic and $\V$ a variety of non-associative algebras over $\K$. If $\V$ has representable representations, then $\V$ satisfies a non-trivial identity of degree two.
\end{proposition}
\begin{proof}
In this proof, we mimic the trick used in Theorem 4.2 of \cite{GM-VdL3}. Indeed, in the proof of Proposition~\ref{two degree char 0}, we wrote $1$ as a linear combination of the $(f_i)_{1\leq i\leq 224}$ in~$\Q[\lambda_1, \dots, \lambda_8, \mu_1, \dots, \mu_8]$. Instead of doing this, we will write some $m\in \N$ as a linear combination in $\Z[\lambda_1, \dots, \lambda_8, \mu_1, \dots, \mu_8]$ of these 224 polynomials because, then, we will just have to check that the system has an empty set of solutions for all infinite fields whose characteristic divides $m$. The problem is that the natural number $m$, equal to $\seqsplit{1456799590845598024309695305537804495463156624065346853298572270548680472045472116250386006867468944668940571897397172623649920632890267296075654343504208784478418777210005902957685588843071241487781773778768300649166665925232915930417490549693708773858134434948706688018655694558517577569557620995746293278480812431412260574404024455598320441859722042018260900473969846828608645627871830598735661629103334242282129060658943436705405397251478428615134881732782202177457769419650118822781315636327042662036615068825734248984068309403420950318}$, which we find with a first computation in \Singular\ is 541 digits long. This makes finding its prime divisors very difficult. The trick is to compute the Gr\"obner basis with a different monomial order. This will give us another natural number $m'$ and we will only need to check the inconsistency of the system of equations for the common prime divisors of $m$ and $m'$. A different monomial order may be chosen in line 1 of the code in Figure~\ref{Code}: for instance, swapping $\mu_7$ and $\mu_8$ is expressed as 
\begin{Verbatim}[numbers=left]
ring r=0,(x(1..8),y(1..6),y(8),y(7)),dp;
\end{Verbatim}
This leads us to the number $m'$ which is equal to $
\seqsplit{52571763195879165827354282293287553627794448211705904729835561552244473038660730349648792264517454512332975448381896476109811545353459315045404360086815158663052807129031203192914881793627758784552281330948466297460892213124611259557307705714671224295558125970110748222333178648754268543416439720181633558932973341496326862122531200555664611204869818358}$ and only 353 digits long. Obviously, $2$ is a common prime divisor. Using the Euclidean algorithm, it is rapidly seen that it is the only one. In order to conclude, we prove as before, using \Singular, that $I = \K[\lambda_1, \dots, \lambda_8, \mu_1, \dots, \mu_8]$ for any field~$\K$ of characteristic $2$, where $I$ is the ideal of $\K [\lambda_1 , \dots , \lambda_8, \mu_1, \dots, \mu_8]$ generated by the polynomials $f_k$ for $k=1$, \dots, $224$. This is done in \Singular\ by using
\begin{Verbatim}[numbers=left]
ring r=2,(x(1..8),y(1..8)),dp;
\end{Verbatim}
in line 1 of the code in Figure~\ref{Code}. Again, explicit linear combinations are given in the arXiv version of this paper as an ancillary file.
\end{proof}

\begin{corollary}\label{Cor Degree Two}
If $\V$ is a variety of non-associative algebras over an infinite field whose representations are representable, then $\V$ is either a subvariety of the variety of commutative algebras, or a subvariety of the variety of anticommutative algebras.
\end{corollary}
\begin{proof}
Combine the above with Proposition~\ref{Comm or Anticomm}.	
\end{proof}

\section{Anticommutativity and the Jacobi identity}\label{Anticommutativity}
The last proposition of the previous section gives us two identities: commutativity and anticommutativity. The goal of this section is to show that representability of representations rules out the first case. Next, we will prove that from anticommutativity and the $\lambda/\mu$-rules (which are both consequences of representability of representations), we can deduce that if the variety is non-abelian, then it has to satisfy the Jacobi identity.

\begin{proposition}\label{AR+comm=abelian}
Let $\V$ be a variety of non-associative algebras over an infinite field of characteristic different from $2$. If $\V$ is a variety of commutative algebras whose representations are representable, then it is an abelian variety.
\end{proposition}
\begin{proof}
First, since $\V$ has representable representations, it is algebraically coherent by Theorem~\ref{RR to AC}. Hence the $\lambda / \mu$-rules hold. By commutativity, these can be rewritten as 
\begin{align}\label{eq co}
x(yz) = \lambda (xy)z + \mu y (xz)
\end{align}
for some $\lambda$, $\mu \in \K$. 

Now we need to use essentially the same algebras and actions as in the proof of Proposition~\ref{two degree char 0}, but without considerations for left and right, so that the actions satisfy the commutativity rule. In other words, for all $b^i$ in $B^i$ and $w$ in $X$ we define the actions such that $bw=wb$. Such actions are well defined, and thus the semi-direct product $(B^1+B^2+B^3)\ltimes X$ is an object of $\V$ by representability of representations. Therefore its elements should satisfy the identity~\eqref{eq co}. Let us check this on $b^1(b^2x)$: indeed, $b^1(b^2x) = \lambda (b^1b^2)x + \mu b^2(b^1x)
= \lambda^2 (xb^1)b^2 +\lambda\mu b^1(xb^2)$, so that $0=(\lambda^2 + \mu) z^{12} + (\lambda\mu -1) z^{21}$. By linear independence of $\{z^{12},z^{21}\}$, either $\V$ is abelian or $\lambda = \mu = -1$.

The second case would mean that the identity~\eqref{eq co} is now
\begin{align}\label{eq co bis}
x(yz) = - (xy)z - y (xz).
\end{align}
Therefore, since $b^1(b^2b^3)$ is an element of the semi-direct product we have that $b^1(b^2b^3)= -(b^1b^2)b^3 - b^2(b^1b^3)$ as an action, and in particular on $x$ this implies
\begin{align*}
(b^1(b^2b^3))x= -((b^1b^2)b^3)x - (b^2(b^1b^3))x.
\end{align*}
Decomposing each element on the left side and on the right side twice by applying~\eqref{eq co bis} gives us $t^{123} + t^{132} + t^{231} + t^{321} = - t^{312} - t^{321} - t^{123} - t^{213} - t^{213} - t^{231} - t^{132} - t^{312}$, which is equivalent to $2t^{312} + 2t^{321} + 2t^{123} + 2t^{213} + 2t^{132} + 2t^{231}= 0$. Since $\kar(\K) \neq 2$ and all these elements are linearly independent, we encounter a contradiction. Hence $\V$ is abelian.
\end{proof}

\begin{remark}
If $\kar(\K)=2$, then commutativity and anticommutativity are the same condition. Therefore, we do not have to care about it to state our next result. However, one should remember that quasi-Lie algebras and Lie algebras do not coincide when the characteristic of the field is $2$.
\end{remark}

\begin{proposition}\label{Jacobi from aco}
Let $\K$ be an infinite field and $\V$ a variety of non-associative algebras over $\K$ that satisfies the $\lambda/\mu$-rules and anticommutativity. If $\V$ has representable representations, then the Jacobi identity holds in $\V$.
\end{proposition}
\begin{proof}
Because of anticommutativity, the $\lambda / \mu$-rules can be rewritten as
\begin{align}\label{eq aco}
x(yz)= \lambda (xy)z + \mu y(xz)
\end{align}
for some $\lambda$, $\mu \in \K$. We consider the actions we defined in the proof of Proposition~\ref{AR+comm=abelian}, corrected for \emph{anti}commutativity. So we decree that an element $b^i$ acting on the left is the same as $b^i$ acting on the right times $-1$. For instance, $b^1y^2=-y^2 b^1$ and $b^2x = -xb^2$. In order to check the identity~\eqref{eq aco} on $b^1(b^2x)$, we compute:
\begin{align*}
b^1(b^2x) =\lambda (b^1b^2)x + \mu b^2(b^1x)
= - \lambda^2 (xb^1)b^2 -\lambda\mu b^1(xb^2) + \mu b^2(b^1x).
\end{align*}
By linear independence, this yields the system $\mu - \lambda^2 =0$, $1-\lambda\mu =0
$ whose solution is $\lambda=\mu =1$. Thus \eqref{eq aco} becomes the Jacobi identity.
\end{proof}

\begin{corollary}\label{corollary subvar of Lie}
Let $\V$ be a variety of non-associative algebras over an infinite field~$\K$. If $\V$ has representable representations, then $\V$ is a subvariety of $\Lie_\K$ or~$\qLie_\K$ if $\kar(\K)=2$, and a subvariety of $\Lie_\K$ otherwise.\noproof
\end{corollary}

\section{What about subvarieties?}\label{Subvarieties}

We concluded the previous section by saying that the only potential identities of degree two for varieties with representable representations are $xy=0$, $xy=-yx$ and $xx=0$, and that the Jacobi identity necessarily holds. The question we answer in this section is: ``Did we miss any other identities?'' As we have already explained above, such identities would have to be of degree at least~$3$. The next proposition proves that we cannot add any non-trivial identities without making the variety abelian.

\begin{proposition}\label{ban subvar}
Let $\K$ be an infinite field and $\V$ be a subvariety of $\Lie_\K$ or $\qLie_\K$ determined by a collection of identities of degree 3 or higher. Then $\V$ is an abelian variety if and only if it has representable representations.
\end{proposition}
\begin{proof}
Let $\V$ be a non-abelian variety of $\K$-algebras with representable representations. In the previous sections it was explained that there are no identities of degree two besides $xy=-yx$ or $xx=0$, and that the Jacobi identity holds. The idea here is to first prove that there is no identity of degree three besides Jacobi, and then we show that no other identities (of degree $n > 3$) can hold.

First, suppose that $\V$ satisfies some degree three identities. Because of anticommutativity and Theorem~\ref{Theorem Linearity General}, we may assume this identity to be $x(yz)= \lambda (xy)z + \mu y(xz)$ for certain $\lambda$, $\mu \in \K$. Then we may recycle the ideas of the proof of Proposition~\ref{Jacobi from aco} to reduce it to the Jacobi identity.

Now, the goal is to prove by induction that the existence of an identity of degree $n>3$ is in contradiction with Theorem~\ref{Theorem Linearity General}. Let us assume that no other identities of degree lower than or equal to $n>3$ are satisfied. Let $\psi$ be a homogeneous identity of degree $n+1$ with $n>3$ which holds in $\V$. We consider a non-trivial multilinear consequence $\varphi(x_1 ,\dots , x_n,x_0)=0$ of $\psi$, whose existence is assured by Theorem~\ref{Theorem Linearity General}. Using anticommutativity and the Jacobi identity, we can rewrite $\varphi$ in the shape $0= \sum_{i=1} \lambda_i \varphi_i (x_1 , \dots, x_i x_0, \dots, x_n)$. We observe that if $\lambda_i=0$ for all $i$, then $\varphi$ is just a consequence of $xy=-yx$ and $x(yz)+y(zx)+z(xy)=0$, or in other words a trivial identity. Note that it is impossible to be a consequence of other identities by the induction hypothesis. Proving that the $\lambda_i$'s are zero will then bring us to a contradiction.

We consider the same actions we used in Proposition~\ref{Jacobi from aco} but extended to $n$ algebras $B^i$ acting on $X$ (which is also enlarged to more elements). Those are again well defined and thus we have that $(B^1 + \dots + B^n) \ltimes X$ lies in $\V$ as already explained. Therefore the element $\psi((b^1,0),(b^2,0), \dots, (b^n,0),(0, x))=\psi(b^1, \dots, b^n,x)$ of this semi-direct product has to vanish and the multilinear consequence $\varphi$ as well, which allows us to write $0= \sum_{i=1} \lambda_i \varphi_i (b^1 , \dots, b^i x, \dots, b^n)$. Again, by construction of the actions, the elements $\varphi_i (b^1 , \dots, b^i x, \dots, b^n)$ of $X$ are linearly independent and thus $\lambda_i=0$ for all $i$---which completes the proof.
\end{proof}

\section{Conclusion and final remarks}\label{Conclusion}

We can now conclude with the main result of this article:

\begin{theorem}\label{AR implies Lie}
Let $\K$ be an infinite field. Let $\V$ be a variety of non-associative algebras over $\K$. If $\V$ has representable representations---which happens, for instance, when it is action representable---then $\V$ is either the variety of Lie algebras~$\Lie_\K$, the variety of quasi-Lie algebras $\qLie_\K$, or the category of vector spaces $\Vect_{\K}$.
\end{theorem}
\begin{proof}
This is a direct consequence of Corollary~\ref{corollary subvar of Lie} and Proposition~\ref{ban subvar}.
\end{proof}

In \cite{GM-VdL3}, the first and third authors of the current article gave a different and \emph{a~priori} unrelated categorical description of Lie algebras, through a condition called \emph{local algebraic cartesian closedness} \LACC\ introduced by J.~R.~A.~Gray~\cite{Gray2012, GrayLie, GM-G}. Let us recall the main result:

\begin{theorem}\label{LACC implies Lie}
Let $\K$ be an infinite field and $\V$ a variety of $n$-algebras over $\K$ which is a non-abelian \emph{locally algebraically cartesian closed} category. Then $n = 2$ and
\begin{enumerate}
\item if $\kar(\K)\neq 2$, then $\V = \Liek = \qLiek$;
\item if $\kar(\K)= 2$, then $\V = \Liek$ or $\V = \qLiek$.\noproof
\end{enumerate}
\end{theorem}

This naturally leads to a number of questions which we hope to investigate in the future. 

\subsection{First question}
Theorem~\ref{AR implies Lie} together with Theorem~\ref{LACC implies Lie} tell us that, for non-associative algebras over an infinite field, action representability, representability of representations and \LACC\ are equivalent conditions. On the other hand, for arbitrary semi-abelian categories, this is known to be false in general. Indeed, the category $\Bool$ of Boolean rings is action representable~\cite{BJK2} but does not satisfy \LACC, as explained in Proposition~6.4 of~\cite{Gray2012}. Today it remains an open problem whether local algebraic cartesian closedness implies action representability/representability of representations, or if some counterexample exists.

\begin{remark}
We observe that $\Bool$ can be seen as the variety of non-associative $\Z_2$-algebras satisfying $xx=x$. The fact that this category is action representable emphasises the necessity in Theorem~\ref{AR implies Lie} of working with an \emph{infinite} field.
\end{remark}

\subsection{Second question}
In order to obtain Theorem~\ref{LACC implies Lie}, the authors of~\cite{GM-VdL3} used a different inconsistent system of polynomial equations. Ours has 224 and may be reduced to a smaller system. Their system consists of 128 polynomials (and can in fact be further reduced as well). Whence, once again, a main result we were not able to prove without computer assistance. This makes us wonder whether a different proof technique exists, preferably a less computationally involved one.

\subsection{Third question}
What if instead of \emph{one} multiplication, algebras have \emph{two}? This is a natural question to ask, since \emph{Poisson algebras} are an important example of this kind of object. Moreover, Poisson algebras over a fixed field form an \emph{Orzech category of interest}, which simplifies the description of actions in this variety. The problem is that, for a given Poisson algebra $X$, the Lie algebra $\Der(X)$ of derivations on $X$ need not form a Poisson algebra in general. Another potentially interesting example is the variety of \emph{Lie--Leibniz algebras} introduced in~\cite{CDL-LL}.

Therefore, we may ask the following questions: ``Are there action representable varieties of non-associative algebras with two multiplications? If not, can we prove this, and generalise the result to $n\geq 2$ multiplications?'' This might either result in a natural, categorical definition of ``Lie algebras with two multiplications'', or in yet another uniqueness result for classical Lie algebras.

\subsection{Fourth question}
What is the scope of representability of representations outside the context of algebras over a field? For instance, how far is it from action representability? What are examples of this condition, what are its consequences, does it admit any interesting characterisations?

\section*{Acknowledgements}
We are grateful to the \emph{Institut de Recherche en Mathématique et Physique} IRMP and the \emph{Department of Mathematics} for their kind hospitality during our stays in Louvain-la-Neuve and in Santiago de Compostela.


\end{document}